\documentclass[10pt]{article}

\usepackage{amsfonts}
\usepackage{amsmath}
\usepackage{indentfirst,latexsym,bm}
\usepackage{amssymb}
\usepackage{amsmath}
\usepackage{graphics}
\usepackage{fancyvrb}
\usepackage{color}
\usepackage{mathrsfs}
\usepackage{hyperref}
\usepackage{amsfonts}
\usepackage{amsmath}
\usepackage{indentfirst,latexsym,bm}
\usepackage{amssymb}
\usepackage{amsmath}
\usepackage{graphics}
\usepackage{authblk}
\usepackage{color}
\usepackage{float}
\usepackage{latexsym}
\usepackage{booktabs}
\usepackage{multirow}
\usepackage{graphicx}
\usepackage{epsfig}
\usepackage{geometry}
\usepackage{amscd}
\usepackage{lscape}
\usepackage{float}
\usepackage{amsmath}
\usepackage{amssymb}
\usepackage{tabularx}
\usepackage{enumerate}
\usepackage{amsfonts}
\usepackage{mathrsfs}
\usepackage{eufrak}
\usepackage{graphicx}
\usepackage{dsfont}
\usepackage{slashbox}
\usepackage{epsfig}
\usepackage[ruled, lined, vlined, shortend, algosection]{algorithm2e}
\usepackage{longtable}
\usepackage{pdflscape}
\usepackage{lscape}
\usepackage{curves}
\usepackage{epic}
\usepackage[table]{xcolor}
\usepackage{booktabs}
\usepackage{subfigure}
\usepackage[all]{xy}

\numberwithin{equation}{section}

\newtheorem{theorem}{Theorem}[section]

\newtheorem{assumption}{Assumption}

\newtheorem{lemma}[theorem]{Lemma}

\newtheorem{proposition}[theorem]{Proposition}

\newenvironment{proof}[1][Proof]{\textbf{#1.} }{\ \rule{0.5em}{0.5em}}
\DeclareMathOperator*{\esssup}{ess\,sup}
\DeclareMathOperator*{\argmax}{arg\,max}
\begin{document}

\title{\textbf{Optimal Switching at Poisson Random Intervention Times}
\footnote{Dedicated to Professor Lishang Jiang for his eightieth
birthday. The authors would like to express their deep gratitude to
Professor Jiang for his supervision when they were his students at
Tongji University. The work is partially supported by a start-up
research fund from King's College London, and the Oxford-Man
Institute, University of Oxford.}}

\author[$a$]{\small\textsc{{Gechun Liang\footnote{Corresponding author. Tel.: +44 020 7848 2633}}}}
\affil[$a$]{\small{\textsc{Department of Mathematics, King's College London, Strand, London, WC2R 2LS, U.K.}}\\

\texttt{gechun.liang@kcl.ac.uk}}\vspace{0.4cm}

\author[$b$]{\small\textsc{Wei Wei}}
\affil[$b$]{\small\textsc{{Mathematical Institute, University of Oxford, 24-29 St Giles, Oxford, OX1 3LB, U.K.}}\\

\texttt{wei.wei@maths.ox.ac.uk}}\vspace{0.4cm}

\maketitle
\begin{abstract}
This paper introduces a new class of optimal switching problems,
where the player is allowed to switch at a sequence of exogenous
Poisson arrival times, and the underlying switching system is
governed by an infinite horizon backward stochastic differential
equation system. The value function and the optimal switching
strategy are characterized by the solution of the underlying
switching system. In a Markovian setting, the paper gives a complete
description of the structure of switching
regions by means of the comparison principle.\\




\end{abstract}
\leftskip0truecm \rightskip0truecm
\newpage
\section{Introduction}

Optimal switching is the problem of determining an optimal sequence
of stopping times for a switching system, which is often modeled by
a stochastic process with several regimes. In this paper, we
introduce and solve a new class of optimal switching problems, where
the player is allowed to switch at a sequence of random times
generated by an exogenous Poisson process, and the underlying
switching system is governed by an infinite horizon backward
stochastic differential equation (BSDE) system.

As a special case of impulse control problems, optimal switching and
its connection with a system of variational inequalities were
extensively studied by Bensoussan and Lions \cite{Bensoussan} in a
Markovian setting, and later by Tang and Yong \cite{TangYong} using
the viscosity solution approach. The problem has drawn renewed
attention recently, due to its various applications in economics and
finance, ranging from firm's investment (see \cite{Oksendal,Zervos})
to real options (see \cite{Carmona, Hamadene1, Touzi}) and trend
following trading (see \cite{Dai}). All applications aim for
determining in an optimal way the sequence of stopping times at
which the player can enter or exit an economic activity, so optimal
switching (in the two-regime case) is also called the \emph{starting
and stopping problem}, or the \emph{reversible investment problem}.
Herein, the player could be the manager of a power plant, who needs
to decide when to produce electricity (if the profit generated from
operation is high), and when to close the power station (if the
profit generated from operation is low) in an optimal way.  For a
more recent development of optimal switching, we refer to Pham
\cite{Pham0} and related references therein.

In this paper, we consider a class of optimal switching problems
where the player is allowed to switch at a sequence of Poisson
arrival times instead of any stopping times. The underlying Poisson
process can be regarded as an exogenous constraint on the player's
ability to switch, so it may reflect the liquidity effect, i.e. the
Poisson process indicates the times at which the underlying system
is available to switch. On the other hand, the Poisson process can
also be seen as an information constraint. The player is allowed to
switch at all times, but she is only able to observe the switching
system at Poisson arrival times. Finally, our model can also be seen
as a randomized version of a discrete optimal switching problem.

In an optimal stopping time setting, a similar problem was firstly
introduced by Dupuis and Wang \cite{Dupuis}, where they used it to
model perpetual American options exercised at exogenous Poisson
arrival times, followed by Lempa \cite{Lempa}. Recently, Liang et al
\cite{Liang2} has established a connection between such kind of
optimal stopping at Poisson random intervention times and dynamic
bank run problems. In this sense, our result in Section 3 is a
generalization of Dupuis and Wang \cite{Dupuis}, Lempa \cite{Lempa}
and Liang et al \cite {Liang2} from optimal stopping to optimal
switching.

There are mainly two methods in the existing literature about how to
solve optimal switching problems. One method mainly focuses on
obtaining closed form solutions, in order to investigate the
structure of switching regions. For example, Brekke and Oksendal
\cite{Oksendal} and Duckworth and Zervos obtained the solutions by a
verification approach. Ly Vath and Pham \cite{LyVath} employed the
viscosity solution approach to determine an explicit solution in the
two-regime case, which was extended to multi-regime case in
\cite{Pham1}. Bayraktar and Egami \cite{Bayraktar} obtained an even
more general result by making an extensive use of one dimensional
diffusions. Notwithstanding, most of the results in this spectrum is
based on the assumption that the switching system is one dimensional
diffusion (such as one dimensional geometric Brownian motion) and
the time horizon is infinite. On the other hand, if the switching
system is governed by a more general stochastic process, such as
multi-dimensional diffusions or BSDEs, it is often a formidable task
to determine the structure of switching regions. In such a
situation, the attention more focuses on the characterization of the
value function and the optimal switching strategy, either by using a
system of variational inequalities as in Tang and Yong
\cite{TangYong}, or by using a system of reflected BSDEs such as in
Hamadene and Jeanblanc \cite{Hamadene1}, Hamadene and Zhang
\cite{Hamadene2}, and Hu and Tang \cite{HuTang}. Numerical solution
is therefore also an important aspect in this method (see Carmona
and Ludkovski \cite{Carmona} and Porchet et al \cite{Touzi}).
Finally, if the underlying switching system is modeled by
non-diffusive processes such as Markov chains, we refer to Bayraktar
and Ludkovski \cite{Bayraktar2}.

In this paper, we try to cover both spectra of the methods to tackle
our optimal switching problem. In Section 2, we present a general
optimal switching model, where the underlying switching system is
governed by an infinite horizon BSDE system. For a general
introduction of BSDEs, we refer to the seminal paper by Pardoux and
Peng \cite{PP}, and a follow-up survey paper by El Karoui et al
\cite{MR1434407}. See also the two monographs by Ma and Yong
\cite{MaYong} and Yong and Zhou \cite{MR1696772} with more
references therein.
Our main result in Section 2 is to show that if the underlying
switching system follows a ``penalized version'' of infinite horizon
BSDE system, then the value of the corresponding optimal switching
problem is nothing but the solution of this penalized equation (see
Theorem \ref{Theorem}). The basic observation comes from the optimal
stopping time representation for one dimensional penalized BSDE,
firstly discovered by Liang \cite{Liang} in a finite horizon case.
In this paper, we also prove an infinite horizon version in Lemma
\ref{lemma1}.

We take the other spectrum of the methods in Section 3, where we
work in a Markovian setting with one dimensional geometric Brownian
motion and two regimes. This simplification enables us to fully
describe the structure of switching regions (see Theorem
\ref{Theorem2}). The basic observation therein is that we can
consider the difference of the value functions for the two switching
regimes, and then employ the comparison principle for one
dimensional equation.\\

The paper is organized as follows: We present our general optimal
switching model in Section 2, and characterize the value of the
optimal switching problem and the associated optimal switching
strategy by the solution of an infinite horizon BSDE system. In
Section 3, we work out a specific example in a Markovian setting,
and give a complete description of the structure of switching
regions. All the technical details are provided in the Appendix.


\section{The Optimal Switching Model}

Let $(W_t)_{t\geq 0}$ be a $n$-dimensional standard Brownian motion
defined on a filtered probability space
$(\Omega,\mathcal{F},\mathbb{F}=\{\mathcal{F}_t\}_{t\geq
0},\mathbf{P})$, where the filtration $\mathbb{F}$ is the minimal
augmented Brownian filtration. For any fixed time $t\geq 0$, let
$\{T_n\}_{n\geq 0}$ be the arrival times of a Poisson process
$(N_s)_{s\geq t}$ with intensity $\lambda$ and minimal augmented
filtration $\{\mathcal{H}_s^{(t,\lambda)}\}_{s\geq t}$. We follow
the convention that $T_0=t$ and $T_{\infty}=\infty$, and throughout
this paper, we assume that the Brownian motion and the Poisson
process are independent. Given the parameter set $(t,\lambda)$, let
$\mathcal{G}_s^{(t,\lambda)}=\mathcal{F}_s\vee\mathcal{H}_s^{(t,\lambda)}$
so that $\mathcal{G}_t^{(t,\lambda)}=\mathcal{F}_t$, and
$\mathbb{G}^{(t,\lambda)}=\{\mathcal{G}_s^{(t,\lambda)}\}_{s\geq
t}$. Moreover, given the Poisson arrival time $T_n$, define
pre-$T_n$ $\sigma$-field:
$$\mathcal{G}_{T_n}^{(t,\lambda)}=\left\{A\in\bigvee_{s\geq t}\mathcal{G}_{s}^{(t,\lambda)}:
A\cap\{T_n\leq s\}\in\mathcal{G}_s^{(t,\lambda)}\ \text{for}\ s\geq
t.\right\}$$ for $n\geq 0$, and denote
$\tilde{\mathbb{G}}^{(t,\lambda)}=\{\mathcal{G}^{(t,\lambda)}_{T_n}\}_{n\geq
0}$.

Let the superscript $^*$ denote the matrix transpose, and $\cdot$
denote the inner product in $\mathbb{R}^d$ with the norm
$|y|=\sqrt{\sum|y_i|^2}$ for $y\in\mathbb{R}^d$. Denote
$||z||=\sqrt{Trace(zz^*)}$ for $z\in\mathbb{R}^{d\times n}$. For
$a\in\mathbb{R}$, let $\mathcal{S}^2_{a}(\mathbb{R}^d)$ be the space
of all $\mathbb{F}$-progressively measurable processes $Y$, valued
in $\mathbb{R}^d$, endowed with the norm:
$$||Y||^2_{\mathcal{S}_a^2}=\mathbf{E}\left[\sup_{t\geq 0}e^{2at}|Y_t|^2\right]<\infty.$$
Let $\mathcal{H}^2_{a}(\mathbb{R}^{d\times n})$ be the space of all
$\mathbb{F}$-progressively measurable processes $Z$, valued in
$\mathbb{R}^{d\times n}$, endowed with the norm:
$$||Z||^2_{\mathcal{H}^2_a}=\mathbf{E}\left[\int_0^{\infty}e^{2at}||Z_t||^2dt\right]<\infty.$$

\subsection{Infinite Horizon BSDE System}

We introduce the following infinite horizon BSDE system, which will
be used to characterize the value of the optimal switching problem
introduced in the next subsection,
\begin{eqnarray}\label{IBSDE1}
\left\{
 \begin{array}{ll}
Y_t^i=Y_T^{i}+\int_t^{T}f_s^i(Y_s,Z_s)+\lambda\max\left\{0,
\mathcal{M}Y_s^i-Y_s^i\right\}ds-\int_t^{T}Z_s^idW_s,\\[+0.2cm]
\lim_{T\uparrow\infty}\mathbf{E}\left[e^{2a T}|Y_T|^2\right]=0
\end{array}
\right.
\end{eqnarray}
for $0\leq t\leq T<\infty$ and $1\leq i\leq d$. The driver
$f_s=(f^1_s\cdots,f_s^d)^*$ and the parameter $a$ are the given
data, and the impulse term $\mathcal{M}Y_t^i$ is defined as
$$\mathcal{M}Y_t^i=\max_{j\neq i}\left\{Y_t^j-g^{ij}_t\right\}.$$
Note that (\ref{IBSDE1}) are nothing but penalized equations of multi-dimensional reflected
BSDEs. A solution to (\ref{IBSDE1}) is a pair of $\mathbb{F}$-progressively
measurable processes $(Y,Z)$ valued in
$\mathbb{R}^d\times\mathbb{R}^{d\times n}$. The solution $Y^i$
represents the payoff in regime $i$, and the impulse term
$\mathcal{M}Y_t^i$ represents the payoff if the player switches from
the current regime $i$ to regime $j$, where $g^{ij}_t$ is the
associate switching cost from $i$ and $j$.

We impose the following assumptions on the data of the infinite
horizon BSDE system (\ref{IBSDE1}).

\begin{assumption}\label{Assumption1} The driver
$f_s(y,z):
\Omega\times\mathbb{R}^+\times\mathbb{R}^d\times\mathbb{R}^{d\times
n }\rightarrow\mathbb{R}^d$ is monotone in $y$ and Lipschitz
continuous in $z$, i.e. there exist constants $a_1,a_2>0$ such that
\begin{align}
(y-\bar{y})^*(f_s(y,z)-f_s(\bar{y},z))&\leq -a_1|y-\bar{y}|^2,\\
|f_s(y,z)-f_s(y,\bar{z})|&\leq a_2||z-\bar{z}||
\end{align}
for any $y,\bar{y}\in\mathbb{R}^d$ and
$z,\bar{z}\in\mathbb{R}^{d\times n}$, and it has linear growth in
both components $(y,z)$. Moreover, the parameter $a$ satisfies the
following structure condition:
\begin{equation}\label{Structure}
a=-a_1+\frac{\delta}{2}a_2^2+(\frac{d+3}{2})\lambda
\end{equation}
for $\delta>1$ such that
\begin{equation}\label{IntegrableCondition}
\mathbf{E}\left[\int_0^{\infty}|f_s(0,0)|^2e^{2a s}ds\right]<\infty.
\end{equation}
\end{assumption}

It is known that structure conditions such as (\ref{Structure}) are
critical for solving infinite horizon BSDE systems. However, the
structure condition (\ref{Structure}) is slightly different from the
standard ones in Section 3 of Darling and Pardoux \cite{Darling} and
Section 2 of Briand and Hu \cite{Briand}: the additional term
$(\frac{d+3}{2})\lambda$ is due to the maximum term
$\lambda\max\left\{0, \mathcal{M}Y_t^i-Y_t^i\right\}$ in
(\ref{IBSDE1}).

The switching cost $g^{ij}$ satisfies the following assumption.
\begin{assumption}\label{Assumption2}
The switching cost $g^{ij}$ for $1\leq i,j\leq d$ is a bounded
$\mathbb{F}$-progressively measurable process valued in
$\mathbb{R}$, and satisfies (1) $g^{ii}_t=0$; (2) $\inf_{t\geq
0}g^{ij}_t+g^{ji}_t\geq C>0$ for $i\neq j$; and (3) $\inf_{t\geq
0}g_t^{ij}+g_t^{jl}-g_t^{il}\geq C>0$ for $i\neq j\neq l$.
\end{assumption}

\begin{proposition}\label{PropositionBSDE}
Suppose that Assumptions \ref{Assumption1} and \ref{Assumption2}
hold. Then there exists a unique solution
$(Y,Z)\in\mathcal{S}^2_{a}(\mathbb{R}^d)\times\mathcal{H}^2_{a}(\mathbb{R}^{d\times
n})$ to the infinite horizon BSDE system (\ref{IBSDE1}).
\end{proposition}

The proof essentially follows from Section 3 of Darling and Pardoux
\cite{Darling} and Section 2 of Briand and Hu \cite{Briand}, though
they consider the random terminal time. For completeness and
readers' convenience, we give the proof of Proposition
\ref{PropositionBSDE} in the Appendix.

We conclude this subsection by observing that $(Y,Z)$ solve
(\ref{IBSDE1}), if and only if the corresponding discounted
processes $(U_t^{i},V_t^{i})=(e^{rt}Y_t^i,e^{rt}Z_t^i)$ for $t\geq
0$, $1\leq i\leq d$ and $r\in\mathbb{R}$ solve the following
infinite horizon BSDE system:
\begin{eqnarray}\label{IBSDE1.0}
\left\{
 \begin{array}{ll}
U_t^{i}=U_T^{i}+\int_t^{T}\tilde{f}_s^{i}(U_s,V_s)+\lambda\max\left\{0,
\tilde{\mathcal{M}}U_t^{i}-U_t^{i}\right\}ds-\int_t^{T}V_s^{i}dW_s,\\[+0.2cm]
\lim_{T\uparrow\infty}\mathbf{E}\left[e^{2(a-r)T}|U_T|^2\right]=0,
\end{array}
\right.
\end{eqnarray}
where the driver
$\tilde{f}_s=(\tilde{f}_s^{1},\cdots,\tilde{f}_s^{d})^*$ is given by
$$\tilde{f}_s^{i}(y,z)=e^{rs}f_s^i(e^{-rs}y,e^{-rs}z)-ry^i$$
for $(y,z)\in\mathbb{R}^d\times\mathbb{R}^{d\times n}$, and the
impulse term $\tilde{\mathcal{M}}U_t^{i}$ is defined as
$$\tilde{\mathcal{M}}U_t^{i}=
\max_{j\neq i}\left\{U_t^{j}-e^{rt}g^{ij}_t\right\}.$$ Hence, as a
direct consequence of Proposition \ref{PropositionBSDE},
(\ref{IBSDE1.0}) admits a unique solution
$(U,V)\in\mathcal{S}^2_{a-r}(\mathbb{R}^d)\times\mathcal{H}^2_{a-r}(\mathcal{R}^{d\times
n})$.

\subsection{Optimal Switching Representation: Main Results}
Consider the following optimal switching problem: Given $d$
switching regimes, a player starts in regime $i\in\{1,\cdots,d\}$ at
any fixed time $t\geq 0$, and makes her switching decisions
sequentially at a sequence of Poisson arrival times $\{T_n\}_{n\geq
0}$ associated with the Poisson process $(N_s)_{s\geq t}$. Hence,
her switching decision at any time $s\geq t$ is represented as
\begin{equation}\label{control}
u_s=\alpha_0\mathbf{1}_{\{t\}}(s)+\sum_{k\geq
0}\alpha_{k}\mathbf{1}_{(T_k,T_{k+1}]}(s),
\end{equation}
where $(\alpha_k)_{k\geq 0}$ is a sequence of
$\mathcal{G}_{T_k}^{(t,\lambda)}$-measurable random variables valued
in $\{1,\cdots,d\}$, so they represent the regime that the player is
going to switching to at the Poisson arrival time $T_k$. Define the
control set $\mathcal{K}_i(t,\lambda)$ as
$$\mathcal{K}_i(t,\lambda)=\left\{
\mathbb{G}^{(t,\lambda)}\text{-progressively measurable process}\
(u_s)_{s\geq t}: u\ \text{has\ the\ form}\ (\ref{control})\
\text{with}\ \alpha_0=i\right\}.$$

The resulting expected payoff associated with any control
$u\in\mathcal{K}_i(t,s)$ is
$$\mathbf{E}\left[\int_t^{\infty}e^{r(s-t)}[f_s^{u_s}(Y_s,Z_s)-rY_s^i]ds-\sum_{k\geq 1}e^{r(T_k-t)}g^{\alpha_{k-1},\alpha_k}_{T_k}|\mathcal{G}^{(t,\lambda)}_t\right]$$
for any $r\leq a$, where the running payoff
$f_s=(f^1_s,\cdots,f^{d}_s)^*$ and the parameter $a$ satisfy
Assumption \ref{Assumption1}, with $(Y,Z)$ given as the solution of
the infinite horizon BSDE system (\ref{IBSDE1}), and the switching
cost $g^{i,j}$ satisfies Assumption \ref{Assumption2}. The player
maximize her expected payoff by choosing an optimal control
$u^*\in\mathcal{K}_i(t,\lambda)$:
\begin{equation}\label{optimalswitching}
y^{i,(t,\lambda)}_t=\esssup_{u\in\mathcal{K}_i(t,\lambda)}
\mathbf{E}\left[\int_t^{\infty}e^{r(s-t)}[f_s^{u_s}(Y_s,Z_s)-rY_s^i]ds-\sum_{k\geq
1}e^{r(T_k-t)}g^{\alpha_{k-1},\alpha_k}_{T_k}|\mathcal{G}^{(t,\lambda)}_t\right].
\end{equation}
Note that if $a\geq 0$, then $r$ can take value zero, and
(\ref{optimalswitching}) corresponds to a non-discounted optimal
switching problem. However, if $a<0$, then the discounting is
necessary, which is not the case for the finite horizon optimal
switching problem.

Our main result is the following representation result of the above
optimal switching problem, which is a counterparty of the finite
horizon case in Section 4 of Liang \cite{Liang}.

\begin{theorem}\label{Theorem}
Suppose that Assumptions \ref{Assumption1} and \ref{Assumption2}
hold. Let $(Y,Z)$ be the unique solution to the infinite horizon
BSDE system (\ref{IBSDE1}). Then the value of the optimal switching
problem (\ref{optimalswitching}) is given by
$y^{i,(t,\lambda)}_t=Y_t^i$, $a.s.$ for $t\geq 0$, and the optimal
switching strategy is  $\tau_0^{*}=t$ and $\alpha_0^{*}=i$,
\begin{equation}\label{optimalswitching1.1}
\tau^*_{k+1}=\inf\left\{T_N> \tau^*_{k}:
Y_{T_N}^{\alpha_k^*}\leq\mathcal{M}Y_{T_N}^{\alpha_k^*}\right\}
\end{equation}
where
$$\alpha_{k+1}^*=\argmax_{j\neq \alpha_{k}^*}\left\{
Y^{j}_{\tau_{k+1}^{*}}-g^{\alpha_k^*,j}_{\tau_{k+1}^{*}}\right\}$$
for $k\geq 0$. Hence, the optimal switching strategy at any time
$s\geq t$ is
$$u^{*}_s=
i\mathbf{1}_{\{t\}}(s)+\sum_{k\geq 0}\alpha_k^{*}
\mathbf{1}_{(\tau_{k}^*,\tau_{k+1}^*]}(s).$$
\end{theorem}

To prove Theorem \ref{Theorem}, we first observe that the solution
$Y^{i}_t$  to (\ref{IBSDE1}) is the value of the optimal switching
problem  (\ref{optimalswitching}) with the associated optimal
switching strategy $u^*$, if and only if the solution
$U^{i}_t=e^{rt}Y^{i}_t$ to (\ref{IBSDE1.0}) is the value of the
following optimal switching problem (without discounting):
\begin{equation}\label{optimalswitching111}
\esssup_{u\in\mathcal{K}_i(t,\lambda)}
\mathbf{E}\left[\int_t^{\infty}\tilde{f}_s^{u_s}(U_s,V_s)ds-\sum_{k\geq
1}e^{rT_k}g^{\alpha_{k-1},\alpha_k}_{T_k}|\mathcal{G}^{(t,\lambda)}_t\right]
\end{equation}
with the optimal switching strategy $\tau^*_0=t$ and $\alpha_0^*=i$,
\begin{equation}\label{optimalswitching111.1}
\tau^*_{k+1}=\inf\left\{T_N> \tau^*_{k}:
U_{T_N}^{\alpha_k^*}\leq\tilde{\mathcal{M}}U_{T_N}^{\alpha_k^*}\right\}
\end{equation}
where
$$\alpha_{k+1}^*=\argmax_{j\neq \alpha_{k}^*}\left\{
U^{j}_{\tau_{k+1}^{*}}-e^{r\tau_{k+1}^*}g^{\alpha_k^*,j}_{\tau_{k+1}^{*}}\right\}.$$

From now on, we will mainly work with the formulation
(\ref{optimalswitching111}), and prove that its value is given by
$U^{i}_t$. The proof crucially depends on the following lemma about
the optimal stopping time representation for the infinite horizon
BSDE system (\ref{IBSDE1.0}), whose proof since quite long, is
postponed to the next subsection. The new feature of this optimal
stopping time problem is that the player is only allowed to stop at
exogenous Poisson arrival times. Such kind of optimal stopping was
first introduced by Dupuis and Wang \cite{Dupuis} in a Markovian
setting.

\begin{lemma}\label{lemma1}
Suppose that Assumptions \ref{Assumption1} and \ref{Assumption2}
hold. Let $(U,V)$ be the unique solution to the infinite horizon
BSDE system (\ref{IBSDE1.0}). For $n\geq 0$ and $1\leq i\leq d$,
consider the following auxiliary optimal stopping time problem:
\begin{equation}\label{optimalstopping}
\tilde{y}_{T_n}^{i,(t,\lambda)}=\esssup_{\tau\in\mathcal{R}_{T_{n+1}}{(t,\lambda)}}
\mathbf{E}\left[\int_{T_n}^{\tau}\tilde{f}_s^{i}(U_s,V_s)ds+
\tilde{\mathcal{M}}U^{i}_{\tau}|\mathcal{G}_{T_n}^{(t,\lambda)}\right],
\end{equation}
where the control set $\mathcal{R}_{T_{n+1}}{(t,\lambda)}$ is
defined as
$$\mathcal{R}_{T_{n+1}}{(t,\lambda)}=\left\{\mathbb{G}^{(t,\lambda)}\text{-stopping\ time}\ \tau\ \text{for}\ \tau(\omega)=T_k(\omega)
\ \text{where}\ k\geq {n+1}\right\}.$$ Then its value is given by
$\tilde{y}_{T_n}^{i,(t,\lambda)}=U^{i}_{T_n}$, $a.s.$ for $n\geq0$,
and in particular, $\tilde{y}_t^{i,(t,\lambda)}=U_t^{i}$, $a.s.$ for
$t\geq 0$. The optimal stopping time is given by
$$\tau_{T_{n+1}}^*=\inf\left\{T_k\geq T_{n+1}: U^{i}_{T_k}\leq \tilde{\mathcal{M}}U^{i}_{T_k}\right\}.$$
\end{lemma}

Let us acknowledge the above lemma for the moment, and proceed to
prove Theorem \ref{Theorem}.


\begin{proof} For any switching strategy
$u\in\mathcal{K}_i(t,\lambda)$ with the form:
$$u_s=i\mathbf{1}_{\{t\}}(s)+\sum_{k\geq
0}\alpha_{k}\mathbf{1}_{(T_k,T_{k+1}]}(s),$$ we consider the
auxiliary optimal stopping time problem (\ref{optimalstopping})
starting from $T_0=t$, stopping at the first Poisson arrival time
$T_1$, and switching to $\alpha_1$:
\begin{equation}\label{firstswitch}
\tilde{y}^{i,(t,\lambda)}_t\geq
\mathbf{E}\left[\int_t^{T_1}\tilde{f}_s^{i}(U_s,V_s)ds+U^{\alpha_1}_{T_1}-e^{rT_1}g^{i,\alpha_1}_{T_1}
|\mathcal{G}_t^{(t,\lambda)}\right].
\end{equation}

Thanks to Lemma \ref{lemma1},  the value of the optimal stopping
time problem (\ref{optimalstopping}) starting from $T_1$ is given by
$\tilde{y}^{\alpha_1,(t,\lambda)}_{T_1}=U^{\alpha_1}_{T_1}$. We
consider such an optimal stopping time problem stopping at the
Poisson arrival time $T_2$, and switching to $\alpha_2$:
\begin{equation}\label{secondswitch}
U^{\alpha_1}_{T_1}=\tilde{y}_{T_1}^{\alpha_1,(t,\lambda)}\geq
\mathbf{E}\left[\int_{T_1}^{T_2}\tilde{f}_s^{\alpha_1}(U_s,V_s)ds+
U^{\alpha_2}_{T_2}-e^{rT_2}g^{\alpha_1,\alpha_2}_{T_2}
|\mathcal{G}_{T_1}^{(t,\lambda)}\right].
\end{equation}
By plugging (\ref{secondswitch}) into (\ref{firstswitch}), we obtain
\begin{equation*}
\tilde{y}_{t}^{i,(t,\lambda)}\geq
\mathbf{E}\left[\int_t^{T_1}\tilde{f}_s^{i}(U_s,V_s)ds+
\int_{T_1}^{T_2}\tilde{f}_s^{\alpha_1}(U_s,V_s)ds-
e^{rT_1}g_{T_1}^{i,\alpha_1}-e^{rT_2}
g_{T_2}^{\alpha_1,\alpha_2}+U_{T_2}^{\alpha_2}|\mathcal{G}_t^{(t,\lambda)}\right].
\end{equation*}
We repeat the above procedure $M$ times, and obtain
\begin{align*}
\tilde{y}_{t}^{i,(t,\lambda)}&\geq
\mathbf{E}\left[\sum_{k=0}^{M}\left(\int_{T_k}^{T_{k+1}}
\tilde{f}_s^{\alpha_{k}}(U_s,V_s)ds-e^{rT_{k+1}}g_{T_{k+1}}^{\alpha_{k},\alpha_{k+1}}
\right)+U^{\alpha_{M+1}}_{T_{M+1}}|\mathcal{G}_t^{(t,\lambda)}\right]\\
&=\mathbf{E}\left[\int_t^{T_{M+1}}\tilde{f}^{u_s}_s(U_s,V_s)ds-\sum_{k=
1}^{M+1}e^{rT_{k}}g_{T_k}^{\alpha_{k-1},\alpha_k}+U^{\alpha_{M+1}}_{T_{M+1}}|\mathcal{G}^{(t,\lambda)}_t\right].
\end{align*}

Since $\mathbf{P}\{\omega: T_{n}(\omega)<\infty\ \text{for\ all}\
n\geq 0\}=1$, the player actually only makes finite number of
switching decisions on $[0,\infty)$, i.e. the switching strategy is
finite. Recall that the solution $U_T$ converges to zero in $L^2$ as
$r\leq a$:
$$\lim_{T\uparrow{\infty}}\mathbf{E}[|U_T|^2]\leq
\lim_{T\uparrow{\infty}}\mathbf{E}[e^{2(a-r)T}|U_T|^2]=
\lim_{T\uparrow{\infty}}\mathbf{E}[e^{2aT}|Y_T|^2]=0.$$ Hence,
letting $M\uparrow\infty$, we get
$$\tilde{y}^{i\,(t,\lambda)}_t\geq\mathbf{E}\left[
\int_t^{\infty}\tilde{f}^{u_s}_s(U_s,V_s)ds-\sum_{k\geq
1}e^{rT_k}g_{T_k}^{\alpha_{k-1},\alpha_k}|\mathcal{G}^{(t,\lambda)}_t\right].$$
Taking the supremum over $u\in\mathcal{K}_i(t,\lambda)$ and using
Lemma \ref{lemma1} once again, we derive that
$$U_t^{i}=\tilde{y}^{i\,(t,\lambda)}_t\geq
\sup_{u\in\mathcal{K}_i(t,\lambda)}\mathbf{E}\left[
\int_t^{\infty}\tilde{f}^{u_s}_s(U_s,V_s)ds-\sum_{k\geq
1}e^{rT_k}g_{T_k}^{\alpha_{k-1},\alpha_k}|\mathcal{G}^{(t,\lambda)}_t\right]
.$$

We prove the reverse inequality by considering the switching
strategy $u=u^*$ as defined in (\ref{optimalswitching111.1}). From
Lemma \ref{lemma1}, $\tau_1^*$ is the optimal stopping time for
$\tilde{y}_t^{i,(t,\lambda)}$. By the definition of $\alpha_{1}^*$,
$$\tilde{\mathcal{M}}U^{i}_{\tau_1^*}=
\max_{j\neq
i}\{U_{\tau_1^*}^{j}-e^{r\tau_1^*}g_{\tau_1^*}^{i,j}\}=U^{\alpha_1^*}_{\tau_1^*}-e^{r\tau_1^*}g_{\tau_1^*}^{i,\alpha_{1}^*}.$$
Therefore,
\begin{equation}\label{firstoptimalswitch}
\tilde{y}^{i,(t,\lambda)}_t=\mathbf{E}\left[
\int_t^{\tau_1^*}\tilde{f}_s^{i}(U_s,V_s)ds+U^{\alpha_1^*}_{\tau_1^*}-e^{r\tau_1^*}g^{i,\alpha_1^*}_{\tau_1^*}
|\mathcal{G}_t^{(t,\lambda)}\right].
\end{equation}
Similarly, $\tau^*_2$ is the optimal stopping time for
$\tilde{y}_{\tau_{1}^*}^{\alpha_1^*,(t,\lambda)}=U_{\tau_{1}^*}^{\alpha_1^*}$,
and
$$\tilde{\mathcal{M}}U^{\alpha_1^*}_{\tau_2^*}=
\max_{j\neq
\alpha_1^*}\{U_{\tau_2^*}^{j}-e^{r\tau_2^*}g_{\tau_2^*}^{\alpha_1^*,j}\}=
U^{\alpha_2^*}_{\tau_2^*}-e^{r\tau_{2}^*}g_{\tau_2^*}^{\alpha_{1}^*,\alpha_2^*}.$$
Hence,
\begin{equation}\label{secondoptimalswitch}
U_{\tau_{1}^*}^{\alpha_1^*}=\tilde{y}^{\alpha_1^*,(t,\lambda)}_{\tau_1^*}=
\mathbf{E}\left[\int_{\tau_1^*}^{\tau_2^*}\tilde{f}_s^{\alpha_1^*}(U_s,V_s)ds+U^{\alpha_2^*}_{\tau_2^*}-
e^{r\tau_2^*}g^{\alpha_1^*,\alpha_2^*}_{\tau_2^*}
|\mathcal{G}_{\tau_1^*}^{(t,\lambda)}\right].
\end{equation}
Plugging (\ref{secondoptimalswitch}) into (\ref{firstoptimalswitch})
gives us
$$\tilde{y}^{i,(t,\lambda)}_t=\mathbf{E}\left[
\int_t^{\tau_1^*}\tilde{f}_s^i(U_s,V_s)ds+\int_{\tau_1^*}^{\tau_2^*}\tilde{f}_s^{\alpha_1^*}(U_s,V_s)ds
-e^{r\tau_1^*}g^{i,\alpha_1^*}_{\tau_1^*}-e^{r\tau_2^*}g^{\alpha_1^*,\alpha_2^*}_{\tau_2^*}
+U^{\alpha_2^*}_{\tau_2^*}|\mathcal{G}_t^{(t,\lambda)}\right]$$

We repeat the above procedure as many times as necessary, and obtain
for any $M\geq 0$,
\begin{align*}
\tilde{y}_{t}^{i,(t,\lambda)}&=
\mathbf{E}\left[\sum_{k=0}^{M}\left(\int_{\tau_k^*}^{\tau^*_{k+1}}
\tilde{f}_s^{\alpha_{k}^*}(U_s,V_s)ds-e^{r\tau_{k+1}^*}g_{\tau^*_{k+1}}^{\alpha_{k}^*,\alpha_{k+1}^*}
\right)+U_{T_{M+1}}^{\alpha_{M+1}}|\mathcal{G}_t^{(t,\lambda)}\right]\\
&=\mathbf{E}\left[
\int_t^{\tau^*_{M+1}}\tilde{f}^{u_s^*}_s(U_s,V_s)ds-\sum_{k=
1}^{M+1}e^{r\tau_k^*}g_{\tau^*_k}^{\alpha_{k-1}^*,\alpha_k^*}+U_{T_{M+1}}^{\alpha_{M+1}}|\mathcal{G}^{(t,\lambda)}_t\right].
\end{align*}

Since the switching strategy is finite and $U_T$ converges to zero
in $L^2$, we conclude $\tilde{y}_{t}^{i,(t,\lambda)}\leq
y_t^{i,(t,\lambda)}$ by taking $M\uparrow\infty$. Then Lemma
\ref{lemma1} implies that $U_t^i=\tilde{y}_{t}^{i,(t,\lambda)}\leq
y_t^{i,(t,\lambda)}$, and $u^*$ is the optimal switching strategy
for the optimal switching problem (\ref{optimalswitching111}).
\end{proof}

\subsection{Optimal Stopping Representation: Proof of Lemma \ref{lemma1}}

The proof is adapted from the proof of Theorem 1.2 in Liang
\cite{Liang}, where a finite horizon problem was considered.

First, we introduce an equivalent formulation of the optimal
stopping time problem (\ref{optimalstopping})

\begin{equation}\label{optimalstopping1}
\tilde{y}_{T_n}^{i,(t,\lambda)}=
\esssup_{N\in\mathcal{N}_{{n+1}}{(t,\lambda)}}
\mathbf{E}\left[\int_{T_n}^{T_{N}}\tilde{f}_s^i(U_s,V_s)ds+
\tilde{\mathcal{M}}U^i_{T_N}|\mathcal{G}_{T_n}^{(t,\lambda)}\right],
\end{equation}
where the control set $\mathcal{N}_{T_{n+1}}{(t,\lambda)}$ is
defined as
$$\mathcal{N}_{{n+1}}{(t,\lambda)}=\left\{\tilde{\mathbb{G}}^{(t,\lambda)}\text{-stopping\ time}\ N\ \text{for}\
N\geq {n+1}\right\}.$$ Notice that (\ref{optimalstopping1}) is a
discrete optimal stopping problem, as the player is allowed to stop
at a sequence of integers $n+1,n+2,\cdots.$  The optimal stopping
time is then some integer-valued random variable $N^{*}_{n+1}$:
$$N^*_{n+1}=\inf\left\{k\geq n+1: U^{i}_{T_k}\leq \tilde{\mathcal{M}}U^i_{T_{k}}\right\}.$$

We will mainly work with the formulation (\ref{optimalstopping1})
from now on. The proof is based on two observations. The first
observation is that the solution to the infinite horizon BSDE system
(\ref{IBSDE1.0}) on the Poisson arrival time $T_n$ can be calculated
recursively as follows:
\begin{equation}\label{recursiveBSDE}
U^i_{T_n}=\mathbf{E}\left[\int_{T_n}^{T_{n+1}}\tilde{f}_s^i(U_s,V_s)ds+\max\{\tilde{\mathcal{M}}U^i_{T_{n+1}},U^i_{T_{n+1}}\}|\mathcal{G}_{T_n}^{(t,\lambda)}\right].
\end{equation}

Indeed, by applying It\^o's formula to $e^{-\lambda t}U_t^{i}$, we
obtain for any $T\geq T_n$,
\begin{align*}
e^{-\lambda T_n}U_{T_n}^{i}=e^{-\lambda
T}U_T^{i}+\int_{T_n}^{T}e^{-\lambda s}\left(\tilde{f}_s^i(U_s,V_s)+
\lambda\max\{\tilde{\mathcal{M}}U_s^{i},U_s^{i}\}\right)ds-\int_{T_n}^{T}e^{-\lambda
s}V_s^{i}dW_s,
\end{align*}
so that
\begin{equation}\label{DPEforBSDE2}
U_{T_n}^{i}=\mathbf{E}\left[e^{-\lambda(T-{T_n})}U_{T}^i+\int_{T_n}^Te^{-\lambda(s-T_n)}\left(\tilde{f}_s^i(U_s,V_s)+
\lambda\max\left\{\tilde{\mathcal{M}}U_s^{i},U_{s}^{i}\right\}\right)ds
|\mathcal{G}_{T_n}^{(t,\lambda)}\right].
\end{equation}

Next, we use integration by parts and the conditional density
$\lambda e^{-\lambda (x-T_n)}dx$ of $T_{n+1}-T_{n}$ to simplify
(\ref{DPEforBSDE2}):
\begin{align*}
&\ \mathbf{E}\left[\int_{T_n}^Te^{-\lambda(s-T_n)}\tilde{f}_s^i(U_s,V_s)ds|\mathcal{G}_{T_n}^{(t,\lambda)}\right]\\
=&\
\mathbf{E}\left[e^{-\lambda(T-T_n)}\int_{T_n}^T\tilde{f}_s^i(U_s,V_s)ds+\int_{T_n}^T\lambda
e^{-\lambda(s-T_n)}(\int_{T_n}^s\tilde{f}_u^i(U_u,V_u)du)ds|\mathcal{G}_{T_n}^{(t,\lambda)}\right]\\
=&\ \mathbf{E}\left[\int_{T_n}^{T_{n+1}\wedge
T}\tilde{f}_s^i(U_s,V_s)ds|\mathcal{G}_{T_n}^{(t,\lambda)}\right].
\end{align*}
Moreover,
\begin{align*}
&\ \mathbf{E}\left[\int_{T_n}^T\lambda
e^{-\lambda(s-{T_n})}\max\left\{\tilde{\mathcal{M}}U_s^i,U_{s}^{i}\right\}ds
+e^{-\lambda(T-T_n)}U_{T}^i|\mathcal{G}_{T_n}^{(t,\lambda)}\right]
\\
=&\
\mathbf{E}\left[\max\left\{\tilde{\mathcal{M}}U_{T_{n+1}}^i,U_{T_{n+1}}^{i}\right\}\mathbf{1}_{\{T_{n+1}<
T\}}+U_{T}^i\mathbf{1}_{\{T_{n+1}\geq
T\}}|\mathcal{G}_{T_n}^{(t,\lambda)}\right].
\end{align*}
Hence, plugging the above two expressions into (\ref{DPEforBSDE2})
gives us
\begin{equation*}
U_{T_n}^{i}=\mathbf{E}\left[\int_{T_n}^{T_{n+1}\wedge
T}\tilde{f}_s^i(U_s,V_s)ds+\max\left\{\tilde{\mathcal{M}}U^i_{T_{n+1}},U_{T_{n+1}}^{i}\right\}\mathbf{1}_{\{T_{n+1}<
T\}}+U^i_{T}\mathbf{1}_{\{T_{n+1}\geq
T\}}|\mathcal{G}_{T_n}^{(t,\lambda)}\right].
\end{equation*}
We conclude (\ref{recursiveBSDE}) by taking $T\uparrow\infty$ in the
above equation.

The second observation is that if we define
$\widehat{U}^{i}=\max\left\{\tilde{\mathcal{M}}U^i,U^{i}\right\}$,
then $\widehat{U}^{i}$ satisfies the following recursive equation:
\begin{equation}\label{DPEFORBSDE3}
\widehat{U}_{T_n}^{i}=\max\left\{\tilde{\mathcal{M}}U^i_{T_{n}},
\mathbf{E}\left[\int_{T_n}^{T_{n+1}}\tilde{f}_s^i(U_s,V_s)ds+\widehat{U}^{i}_{T_{n+1}}|\mathcal{G}_{T_n}^{(t,\lambda)}\right]\right\},
\end{equation}
We show that
$\left(\int_0^{T_n}\tilde{f}^i_s(U_s,V_s)ds+\widehat{U}^i_{T_n}\right)_{n\geq
0}$ is the snell envelop of
$\left(\int_0^{T_n}\tilde{f}^i_s(U_s,V_s)ds+\tilde{\mathcal{M}}U^i_{T_n}\right)_{n\geq
0}$ in the following lemma.

\begin{lemma}\label{lemma2}
Suppose that Assumptions \ref{Assumption1} and \ref{Assumption2}
hold. Let $(U,V)$ be the unique solution to the infinite horizon
BSDE system (\ref{IBSDE1.0}). Then the value of the following
optimal stopping time problem
\begin{equation}\label{optimalstopping2}
\widehat{y}_{T_n}^{i,(t,\lambda)}=
\esssup_{N\in\mathcal{N}_{{n}}{(t,\lambda)}}
\mathbf{E}\left[\int_{T_n}^{T_N}\tilde{f}_s^i(U_s,V_s)ds+
\tilde{\mathcal{M}}U^i_{T_N}|\mathcal{G}_{T_n}^{(t,\lambda)}\right]
\end{equation}
is given by $\widehat{y}_{T_n}^{i,(t,\lambda)}=\hat{U}^i_{T_n}$
$a.s.$ for $n\geq 0$. The optimal stopping time is given by
$$N^{*}_{n}=
\inf\left\{k\geq n: U^{i}_{T_k}\leq
\tilde{\mathcal{M}}U^{i}_{T_k}\right\} = \inf\left\{k\geq n:
\widehat{U}^{i}_{T_k}\leq \tilde{\mathcal{M}}U^{i}_{T_k}\right\}.$$
\end{lemma}

\begin{proof} Without loss of generality, we may assume
$\tilde{f}^i_s(U_s,V_s)=0$. Otherwise we only need to consider
$\int_0^{T_n}\tilde{f}^i_s(U_s,V_s)ds+\widehat{y}_{T_n}^{i,(t,\lambda)}$
and
$\int_0^{T_n}\tilde{f}_s^i(U_s,V_s)ds+\tilde{\mathcal{M}}U^i_{T_n}$
instead of $\widehat{y}_{T_n}^{i,(t,\lambda)}$ and
$\tilde{\mathcal{M}}U^i_{T_n}$.

From (\ref{DPEFORBSDE3}), $(\widehat{U}^i_{T_m})_{m\geq n}$ is
obviously a $\tilde{\mathbb{G}}^{(t,\lambda)}$-supermartingale.
Hence, for any $N\in\mathcal{N}_n(t,\lambda)$,
$$\widehat{U}^i_{T_n}\geq \mathbf{E}\left[\widehat{U}^i_{T_N}|\mathcal{G}^{(t,\lambda)}_{T_n}
\right]\geq
\mathbf{E}\left[\tilde{\mathcal{M}}U^i_{T_{N}}|\mathcal{G}^{(t,\lambda)}_{T_n}
\right],
$$
where we used (\ref{DPEFORBSDE3}) in the second inequality. Taking
the supremum over $N\in\mathcal{N}_n(t,\lambda)$, we get
$\widehat{U}^i_{T_n}\geq \widehat{y}_{T_n}^{i,(t,\lambda)}$.

To prove the reverse inequality, we first show that
$\left(\widehat{U}_{T_{m\wedge N^*_n}}^{i}\right)_{m\geq n}$ is a
$\tilde{\mathbb{G}}^{(t,\lambda)}$-martingale. Indeed,

\begin{align*}
\mathbf{E}\left[\widehat{U}^{i}_{T_{m+1\wedge
N^*_n}}|\mathcal{G}_{T_m}^{(t,\lambda)}\right]&
=\mathbf{E}\left[\left(\sum_{j=n}^m\mathbf{1}_{\{{N}_n^*=j\}}+\mathbf{1}_{\{{N}_n^*\geq
m+1\}}\right)\widehat{U}^{i}_{T_{m+1\wedge{N}^*_n}}|\mathcal{G}_{T_m}^{(t,\lambda)}\right]\\
&=\mathbf{E}\left[\sum_{j=n}^m\mathbf{1}_{\{{N}_n^*=j\}}\widehat{U}_{T_{j}}^{i}
+\mathbf{1}_{\{{N}_n^*>m\}}\widehat{U}^{i}_{T_{m+1}}|\mathcal{G}_{T_m}^{(t,\lambda)}\right]\\
&=\sum_{j=n}^m\mathbf{1}_{\{{N}_n^*=j\}}\widehat{U}_{T_{j}}^{i}+
\mathbf{1}_{\{{N}_n^*>m\}}\mathbf{E}\left[\widehat{U}^{i}_{T_{m+1}}|\mathcal{G}_{T_{m}}^{(t,\lambda)}\right]\\
&=\sum_{j=n}^m\mathbf{1}_{\{{N}_n^*=j\}}\widehat{U}_{T_{j}}^{i}+
\mathbf{1}_{\{{N}_n^*>m\}}\widehat{U}^{i}_{T_m}=\widehat{U}^{i}_{T_{m\wedge
{N}_n^*}},
\end{align*}
where we used (\ref{DPEFORBSDE3}) and the definition of ${N}^*_n$ in
the second last equality. Hence,
\begin{align*}
\widehat{U}^i_{T_n}=\mathbf{E}[\widehat{U}^i_{
T_{N^*_n}}|\mathcal{G}^{(t,\lambda)}_{T_n}]=
\mathbf{E}[\tilde{\mathcal{M}}U^i_{ T_{
N^*_n}}|\mathcal{G}^{(t,\lambda)}_{T_n}]\leq
\widehat{y}_{T_n}^{i,(t,\lambda)},
\end{align*}
and ${N}^*_n$ is the optimal stopping time for the optimal stopping
problem (\ref{optimalstopping2}).
\end{proof}

We are now in a position to prove Lemma \ref{lemma1}. From
(\ref{recursiveBSDE}) and the definition of $\hat{U}$,
\begin{equation}\label{recursiveBSDE1}
U^i_{T_n}=\mathbf{E}\left[\int_{T_n}^{T_{n+1}}\tilde{f}_s^i(U_s,V_s)ds+\widehat{U}^i_{T_{n+1}}|\mathcal{G}_{T_n}^{(t,\lambda)}\right].
\end{equation}
Thanks to Lemma \ref{lemma2},
$\widehat{U}^i_{T_{n+1}}=\widehat{y}^{i,(t,\lambda)}_{T_{n+1}}$,
which is the value of the optimal stopping problem
(\ref{optimalstopping2}). Hence, for any
$N\in\mathcal{N}_{n+1}(t,\lambda)$,
\begin{equation}\label{optimalstopping2.1}
\widehat{U}^i_{T_{n+1}}=\widehat{y}_{T_{n+1}}^{i,(t,\lambda)}\geq
\mathbf{E}\left[\int_{T_{n+1}}^{T_N}\tilde{f}_s^i(U_s,V_s)ds+
\tilde{\mathcal{M}}U^i_{T_N}|\mathcal{G}_{T_{n+1}}^{(t,\lambda)}\right].
\end{equation}

By plugging (\ref{optimalstopping2.1}) into (\ref{recursiveBSDE1}),
we obtain
$$U^i_{T_n}\geq \mathbf{E}\left[\int_{T_n}^{T_{n+1}}\tilde{f}_s^i(U_s,V_s)ds+
\int_{T_{n+1}}^{T_N}\tilde{f}_s^i(U_s,V_s)ds+\tilde{\mathcal{M}}U^i_{T_N}|\mathcal{G}_{T_n}^{(t,\lambda)}\right].$$
Taking the supremum over $N\in\mathcal{N}_{n+1}(t,\lambda)$ gives us
$U^i_{T_n}\geq\tilde{y}^{i,(t,\lambda)}_{T_n}$.

To prove the reverse inequality, we take $N^*_{n+1}=\inf\{k\geq n+1:
U_{T_{k}}^i\leq\tilde{\mathcal{M}}U_{T_k}^i\}$, which is the optimal
stopping time for $\widehat{y}_{T_{n+1}}^{i,(t,\lambda)}$, and
therefore,
\begin{equation}\label{optimalstopping2.2}
\widehat{U}^i_{T_{n+1}}=\widehat{y}_{T_{n+1}}^{i,(t,\lambda)}=
\mathbf{E}\left[\int_{T_{n+1}}^{T_{N^*_{n+1}}}\tilde{f}_s^i(U_s,V_s)ds+
\tilde{\mathcal{M}}U^i_{T_{N^*_{n+1}}}|\mathcal{G}_{T_{n+1}}^{(t,\lambda)}\right].
\end{equation}

Plugging (\ref{optimalstopping2.2}) into (\ref{recursiveBSDE1})
gives us
$$U^i_{T_n}=\mathbf{E}\left[\int_{T_n}^{T_{n+1}}\tilde{f}_s^i(U_s,V_s)ds+
\int_{T_{n+1}}^{T_{N^*_{n+1}}}\tilde{f}_s^i(U_s,V_s)ds+\tilde{\mathcal{M}}U^i_{T_{N^*_{n+1}}}|\mathcal{G}_{T_n}^{(t,\lambda)}\right]\leq
\tilde{y}^{i,(t,\lambda)}_{T_n}.$$ Hence the value of the optimal
stopping time problem (\ref{optimalstopping}) is given by
$\tilde{y}^{i,(t,\lambda)}_{T_n}=U^i_{T_n}$, and the optimal
stopping time is $N^*_{n+1}$, or equivalently, $\tau^*_{T_{n+1}}$.
\section{The Structure of Switching Regions in a Markovian Case}

In this section, we investigate the switching regions of the optimal
switching problem (\ref{optimalstopping}) in a Markovian setting.
Specifically, we assume that there are two switching regimes $d=2$,
and the Brownian motion is one dimension $n=1$. Moreover,
Assumptions \ref{Assumption1} and \ref{Assumption2} are replaced by
\begin{assumption}\label{Assumption3}
The driver $f_s(y,z)$ has the form: $f_s(y,z)=h(X_s)-a_1 y,$ where
$X$ is a geometric Brownian motion starting from
$X_0=x\in\mathbb{R}^+$ with constant drift $b$ and constant
volatility $\sigma>0$:
$$dX_s=bX_sds+\sigma X_sdW_s,$$
$h=(h^1,h^2)^*$ is nonnegative and Lipschitz continuous, and
$a_1>\max\{b,0\}$ is large enough so that for $a=-a_1+2.5\lambda$,
$$\mathbf{E}\left[\int_0^{\infty}|h(X_s)|^2e^{2as}ds\right]<\infty.$$
\end{assumption}
\begin{assumption}\label{Assumption4}
The switching cost $g^{ij}$ for $i,j\in\{1,2\}$ is a constant, and
satisfies (1) $g^{ii}=0$; and (2) $g^{ij}+g^{ji}>0$ for $i\neq j$.
\end{assumption}

Under Assumptions \ref{Assumption3} and \ref{Assumption4}, the
solution to the infinite horizon BSDE system (\ref{IBSDE1}) is
Markovian, i.e. there exist measurable functions $v=(v^1,v^2)$ such
that $Y_t=v(X_t)$. By Proposition \ref{PropositionBSDE}, $v(X_t)$ is
the solution to the following equation:
\begin{eqnarray}\label{IBSDE3}
\left\{
 \begin{array}{ll}
v^i(X_t)=v^{i}(X_T)+\int_t^{T}h^i(X_s)-a_1
v^i(X_s)+\lambda\max\left\{0,
v^{j}(X_s)-g^{ij}-v^i(X_s)\right\}ds-\int_t^{T}Z_s^idW_s,\\[+0.2cm]
\lim_{T\uparrow\infty}\mathbf{E}\left[e^{2a T}|v(X_T)|^2\right]=0
\end{array}
\right.
\end{eqnarray}
for $i=1,2$ and $j=3-i$. Without loss of generality, we may also
assume that $h(0)=0$. Indeed, if not, we consider
$\tilde{v}^i(X_t)=v^i(X_t)-\frac{h^i(0)}{a_1}$,
$\tilde{h}^i(X_t)=h^i(X_t)-h^i(0)$, and
$\tilde{g}^{ij}=g^{ij}+\frac{h^i(0)-h^j(0)}{a_1}$. Then,
\begin{eqnarray*}
\left\{
 \begin{array}{ll}
\tilde{v}^i(X_t)=\tilde{v}^{i}(X_T)+\int_t^{T}\tilde{h}^i(X_s)-a_1
\tilde{v}^i(X_s)+\lambda\max\left\{0,
\tilde{v}^{j}(X_s)-\bar{g}^{ij}-\tilde{v}^i(X_s)\right\}ds-\int_t^{T}Z_s^idW_s,\\[+0.2cm]
\lim_{T\uparrow\infty}\mathbf{E}\left[e^{2a
T}|\tilde{v}(X_T)|^2\right]=0.
\end{array}
\right.
\end{eqnarray*}

From Theorem \ref{Theorem}, by choosing $r=-a_1$, we know that
$v^i(X_0)=v^i(x)$ is the value of the following optimal switching
problem:
\begin{equation}\label{optimalswitching3}
v^i(x)=\esssup_{u\in\mathcal{K}_i(0,\lambda)}
\mathbf{E}\left[\int_0^{\infty}e^{-a_1s}h^{u_s}(X_s)ds-\sum_{k\geq
1}e^{-a_1T_k}g^{\alpha_{k-1},\alpha_k}\right].
\end{equation}
Moreover, the optimal switching strategy is given as
(\ref{optimalswitching1.1}): $\tau_0^{*}=0$ and $\alpha_0^*=i$,
\begin{equation*}
\tau^*_{k+1}=\inf\left\{T_N> \tau^*_{k}: v^{\alpha_k^*}(X_{T_N})\leq
v^{\alpha_{k+1}^*}(X_{T_N})-g^{\alpha_k^*,\alpha^*_{k+1}}\right\}
\end{equation*}
where $\alpha_{k+1}^*=3-\alpha_{k}^*$.

Therefore, the player will switch from regime $i$ to regime $j$, if
$X$ is in the following switching region $\mathcal{S}^i$ at Poisson
arrival times:
\begin{equation*}
\mathcal{S}^i = \{ x\in(0,\infty): v^i(x) \leq v^j(x)-g^{ij}\}
\end{equation*}
for $i=1,2$ and $j=3-i$. On the other hand, the player will stay in
regime $i$, if $X$ is in the following continuation region
$\mathcal{C}^i$:
\begin{equation*}
\mathcal{C}^i = \{ x\in(0,\infty): v^i(x) > v^j(x)-g^{ij}\}.
\end{equation*}
We further set
\begin{align*}
\underline{x}^i&=\inf\mathcal{S}^i\in [0,\infty];\\
\overline{x}^i&=\sup\mathcal{S}^i\in [0,\infty]
\end{align*}
with the usual convention that $\inf\phi=\infty$ and $\sup\phi=0$.

To distinguish the regimes $1$ and $2$, we impose the following
assumption, which includes several interesting cases for
applications.

\begin{assumption}{\label{s3}}
The difference of the running profits is non-negative: $F(x)=
h^2(x)-h^1(x)\geq 0$, and is strictly increasing on $(0,\infty)$,
and moreover, the switching cost from regime $1$ to regime $2$ is
positive: $g_{12}>0$.
\end{assumption}

The above assumption has clear financial meanings. The
non-negativity means that the regime $2$ is more favorable than the
regime $1$. The monotonicity implies that the improvement is better
and better. Thus, it is natural to assume that the corresponding
switching cost $g^{12}$ from regime $1$ to regime $2$ is positive.

The main result of this section is the following characterization of
the switching regions of the optimal switching problem
(\ref{optimalswitching3}), which are also demonstrated in Figure 1.

\begin{theorem}\label{Theorem2}
Suppose that Assumptions \ref{Assumption3}, \ref{Assumption4} and
\ref{s3} hold. Then we have the following five cases for the
switching regions $\mathcal{S}^1$ and $\mathcal{S}^2$:
\begin{enumerate}
\item If $a_1g^{12}\geq F(\infty)$ and $a_1g^{21}\geq 0$, then,
$\mathcal{S}^1=\phi, \mathcal{S}^2=\phi$;
\item If $a_1g^{12}\geq F(\infty)$ and $-F(\infty)<a_1g^{21}< 0$, then, $\mathcal{S}^1=\phi, \mathcal{S}^2=(0,\overline{x}^2]$, where $\overline{x}^2\in(0,\infty)$;
\item If $a_1g^{12}\geq F(\infty)$ and $a_1g^{21}\leq -F(\infty)$, then, $\mathcal{S}^1=\phi, \mathcal{S}^2=(0,\infty)$;
\item If $a_1g^{12}< F(\infty)$ and $a_1g^{21}\geq 0$, then, $\mathcal{S}^1=[\underline{x}^1,\infty),
\mathcal{S}^2=\phi$, where $\underline{x}^1\in(0,\infty)$.
\item If $a_1g^{12}< F(\infty)$ and $ -F(\infty)<a_1g^{21}<0$, then, $\mathcal{S}^1=[\underline{x}^1,\infty), \mathcal{S}^2=(0,\overline{x}^2]$, where
$\underline{x}^1,\overline{x}^2\in(0,\infty)$.
\end{enumerate}
\end{theorem}

\subsection{The Structure of Switching Regions: Proof of Theorem \ref{Theorem2}}

The proof relies on several basic properties of the value function
$v(\cdot)$, and the associated comparison principle. First, we prove
that the value function $v(\cdot)$ has at most linear growth and is
Lipschitz continuous by employing the optimal switching
representation (\ref{optimalswitching3}). The proof is adapted from
Ly Vath and Pham \cite{LyVath}, and is provided in the Appendix.

\begin{proposition}\label{Propositionlineargrowth}
Suppose that Assumptions \ref{Assumption3} and \ref{Assumption4}
hold. Then the value function $v(\cdot)$ of the optimal switching
problem has at most linear growth,
$$|v^i(x)|\leq C(1+x),$$
and is Lipschitz continuous,
$$|v^i(x)-v^i(\bar{x})|\leq C|x-\bar{x}|,$$
for $i=1,2$ and $x,\bar{x}\in\mathbb{R}^+$.
\end{proposition}

Given the above linear growth property and the Lipschitz continuous
property of $v(\cdot)$, the nonlinear Feynman-Kac formula (see
Section 6 of \cite{Darling}) then implies that $v(\cdot)$ is the
unique (viscosity) solution to the following system of ODEs:
\begin{equation}\label{ode}
-\mathcal{L}v^i(x)+a_1v^i(x)=\lambda\max\left\{0,
v^j(x)-g^{ij}-v^i(x)\right\}+h^i(x)
\end{equation}
for $i=1,2$ and $j=3-j$, where the operator
$\mathcal{L}=\frac12\sigma^2x^2\frac{d^2}{d x^2}+bx\frac{d}{dx}$. Note that (\ref{ode}) are actually
penalized equations for the system of variational inequalities:
$$\min\{-\mathcal{L}v^i(x)+a_1v^i(x)-h^i(x), v^i(x)-v^j(x)+g^{ij}\}=0.$$
Moreover, the following comparison principle for (\ref{ode}) also holds, whose proof
is also provided in the Appendix for completeness.

\begin{proposition}\label{propositioncompare}
Suppose that $\underline{v}^i(\cdot)$ is the subsolution of the
following ODE:
\begin{equation}\label{ode1}
-\mathcal{L}\underline{v}^i(x)+a_1\underline{v}^i(x)\leq 0
\end{equation} for
$x\in\mathbb{R}^+$, and that $\underline{v}^i(\cdot)$ has at most
linear growth. Then $\underline{v}^i(x)\leq 0$ for
$x\in\mathbb{R}^+$. The supersolution property also holds for
$\overline{v}^i(\cdot)$ in a similar way.
\end{proposition}

\noindent\textbf{The structure of the switching region
$\mathcal{S}^1$}

Define $G_1(x)=v^1(x)-v^2(x)+g^{12}$. From (\ref{ode}), it is easy
to see that $G_1(\cdot)$ is the solution of the following ODE:
\begin{equation}\label{s4}
-\mathcal{L}G_1(x)+a_1G_1(x)=\lambda(-G_1(x))^+-\lambda(G_1(x)-g^{12}-g^{21})^++a_1g^{12}-F(x).
\end{equation}
From Proposition \ref{Propositionlineargrowth}, $G_1(\cdot)$ has at
most linear growth and is Lipschitz continuous. The switching region
$\mathcal{S}^1=\{x\in(0,\infty): G_1(x)\leq 0\}$.

\begin{lemma}\label{l1} Suppose that Assumptions \ref{Assumption3},
\ref{Assumption4} and \ref{s3} hold. If there exists
$x^0\in\mathcal{S}^1$, then $[x^0,\infty)\in\mathcal{S}^1$. Hence,
if $\mathcal{S}^1\neq\phi$, then $\bar{x}^1=\infty$.
\end{lemma}

\begin{proof}
Since $G_1(\cdot)$ is Lipschitz continues, its derivative
${G_1'}(x)=\frac{d}{dx}G_1(x)$ is bounded. Moreover. from
(\ref{s4}), it is easy to see ${G_1'}(\cdot)$ is the subsolution of
the following ODE:
\begin{equation}\label{s5}
-\frac{1}{2}\sigma^2 x^2 \frac{d^2 G_1'}{dx^2}-(b+\sigma^2)
x\frac{dG_1'}{dx}+(a_1-b+\lambda H(-G_1)+\lambda
H(G_1-g^{12}-g^{21}))G_1'=-F'(x)\leq 0,
\end{equation}
where $H(x)=\mathbf{1}_{[0,\infty)}(x)$. Since $a_1>b$, by using the
subsolution property in Proposition \ref{propositioncompare}, we
obtain $G_1'(x)\leq 0$. Thus, if $G_1(x^0)\leq 0$, then $G_1(x)\leq
0$ for $x\in[x^0,\infty)$.
\end{proof}

\begin{proposition}\label{p1} Suppose that Assumptions \ref{Assumption3}, \ref{Assumption4} and
\ref{s3} hold. Then
\begin{enumerate}
\item If $a_1g^{12}\geq F(\infty)$, then $\underline{x}^1=\infty$, i.e. $\mathcal{S}^1 = \phi$.
\item If $a_1g^{12}<F(\infty)$, then $\underline{x}^1\in(0,\infty)$ and $\mathcal{S}^1=[\underline{x}^1,\infty)$.
\end{enumerate}
\end{proposition}

\begin{proof}
\begin{enumerate}

\item We claim that $\underline{x}^1 \neq 0$. If not, then
$\underline{x}^1=\inf\mathcal{S}^1=0$, so that $G_1(0)\leq 0$ by the
continuity of $G_1(\cdot)$ on $(0,\infty)$. On the other hand, from
(\ref{s4}), we have $G_1(0)=\dfrac{a_1g^{12}}{a_1+\lambda}>0$, which
is a contradiction.

Now we show that $\underline{x}^1 = \infty$. Suppose not. Then
$0<\underline{x}^1=\inf\mathcal{S}^1<\infty$. Due to the continuity
of $G_1(\cdot)$ on $(0,\infty)$, we have $G_1(\underline{x}^1)=0$,
so $\underline{x}^1\in\mathcal{S}^1$. From Lemma \ref{l1},
$[\underline{x}^1,\infty)\in\mathcal{S}^1$, and $G_1(x)\leq 0$ for
$x\in[\underline{x}^1,\infty)$. Therefore, (\ref{s4}) reduces to
\begin{equation*}
-\mathcal{L}G_1(x)+(a_1+\lambda)G_1(x)=a_1g^{12}-F(x)>0
\end{equation*}
on $[\underline{x}^1,\infty)$. But from the supersolution property
in Proposition \ref{propositioncompare}, we get $G_1(x)>0$ for
$x\in[\underline{x}^1,\infty)$, which is obscured.

\item
Similar to the proof in 1, we have $\underline{x}^1 \neq 0$.
Therefore, we only need to prove $\underline{x}^1 \neq \infty$. If
not, then $G_1(x)>0$ for $x\in(0,\infty)$, and (\ref{s4}) reduces to
\begin{equation*}
-\mathcal{L}G_1(x)+a_1G_1(x)=-\lambda(G_1(x)-g^{12}-g^{21})^++a_1g^{12}-F(x).
\end{equation*}

By the comparison principle in Proposition \ref{propositioncompare},
we have $G_1(x)\leq G(x)$, where $G(x)$ is the solution with linear
growth to the following ODE:
\begin{equation*}
-\mathcal{L}G(x)+a_1G(x)=a_1g^{12}-F(x).
\end{equation*}
The Feynman-Kac formula implies that
\begin{equation*}
G(x) = \mathbf{E}[\int^{\infty}_0 e^{-a_1t}(a_1g^{12}-F(X^x_t)) dt]
\end{equation*}
By Fatou lemma, we have
\begin{equation*}
G(\infty) = \varlimsup_{x\to\infty}\mathbf{E}[\int^{\infty}_0
e^{-a_1t}(rg^{12}-F(X^x_t)) dt] \leq
\mathbf{E}[\int^{\infty}_0e^{-a_1t}(rg^{12}-F(\infty))dt]<0.
\end{equation*}
Thus, we have $G_1(\infty)<0$. This is a contradiction.
\end{enumerate}
\end{proof}

The financial intuition behind Proposition \ref{p1} is that the
structure of the switching region of the regime $1$ depends on the
``instant loss'' of the switching cost $g^{12}$ and the ``net
running profit'' $F=h^2-h^1$. If $a_1g^{12}\geq F(\infty)$, which
means the loss by switching can not be compensated by the net
running profit, then one has no interest to switch; If
$a_1g^{12}<F(\infty)$, which means the net running profit may exceed
the loss due to the switching cost in some state, then one may
switch when the net running profit reaches some level at Poisson arrival times.\\

\noindent\textbf{The Structure of the switching region
$\mathcal{S}^2$}

Define $G_2(x)=v^2(x)-v^1(x)+g^{21}$, which is the solution to the
following ODE:
\begin{equation}\label{s6}
-\mathcal{L}G_2(x)+a_1G_2(x)=\lambda(-G_2(x))^+-\lambda(G_2(x)-g^{21}-g^{12})^++a_1g^{21}+F(x).
\end{equation}
From Proposition \ref{Propositionlineargrowth}, $G_2(\cdot)$ has at
most linear growth and is Lipschitz continuous. The switching region
$\mathcal{S}^2=\{x\in(0,\infty): G_2(x)\leq 0\}$.

\begin{lemma}\label{l2}
Suppose that Assumptions \ref{Assumption3}, \ref{Assumption4} and
{\ref{s3}} hold. If there exists $x^0\in\mathcal{S}^2$, then
$(0,x^0]\in\mathcal{S}^2$. Hence, if $\mathcal{S}^2\neq\phi$, then
$\underline{x}^2=0$.
\end{lemma}

\begin{proof}
From Lemma \ref{l1}, we have $G_2'(x)=-G_1'(x)\geq0$. Thus, if
$G_2(x^0)\leq 0$, then $G_2(x)\leq 0$ for  $x\in(0,x^0]$.
\end{proof}

\begin{proposition}\label{p2}
Suppose that Assumptions \ref{Assumption3}, \ref{Assumption4} and
\ref{s3} hold. Then
\begin{enumerate}
\item If $a_1g^{21}\geq 0$, then $\overline{x}^2=0$, i.e. $\mathcal{S}^2 = \phi$.
\item If $-F(\infty)<a_1g^{21}<0$, then $\overline{x}^2\in(0,\infty)$ and $\mathcal{S}^2=(0,\overline{x}_2]$.
\item If $a_1g^{21}\leq-F(\infty)$, then $\mathcal{S}^2=(0,\infty)$.
\end{enumerate}
\end{proposition}

\begin{proof}
\begin{enumerate}
\item
If there exists $\overline{x}^2\in(0,\infty]$, then $G_{2}(0)< 0$
(From (\ref{s6}), it is easy to see that there does not exist any
interval in which $G_2$ keeps constant.). From (\ref{s6}), we can
see that $G_{2}(0) = \dfrac{a_1g^{21}}{a_1+\lambda}\geq 0$. This is
a contradiction, so $\bar{x}^2=\sup\mathcal{S}^2$ must be zero.

\item
Thanks to Lemma \ref{l2}, we only need to show that
$\overline{x}^2\neq 0$ and $\overline{x}^2\neq \infty$.

If $\overline{x}_2=0$, then $G_2(x)>0$ for $x\in(0,\infty)$, and
$G_2(0)\geq 0$ due to the continuity of $G_2(\cdot)$. Hence,
(\ref{s6}) reduces to
\begin{equation*}
-\mathcal{L}G_2(x)+a_1G_2(x)=-\lambda(G_2(x)-g^{21}-g^{12})^++a_1g^{21}+F(x).
\end{equation*}
By the comparison principle in Proposition \ref{propositioncompare},
we have $G_2(x)\leq G(x)$, where $G(x)$ is the solution with linear
growth to the following ODE:
\begin{equation*}
-\mathcal{L}G(x)+a_1G(x)=a_1g^{21}+F(x).
\end{equation*}
By continuity of both $G_2(\cdot)$ and $G(\cdot)$, $G_2(0)\leq
G(0)=g_{21}<0$, which is a contradiction. Thus, we have
$\overline{x}^2\neq 0$.

If $\overline{x}_2=\infty$, i.e. $G_2(x)\leq 0$ for
$x\in(0,\infty)$, then (\ref{s6}) reduces to
\begin{equation*}
-\mathcal{L}G_2(x)+(a_1+\lambda)G_2(x)=a_1g^{21}+F(x).
\end{equation*}
The Feynman-Kac formula implies that
\begin{equation*}
G_2(x) = \mathbf{E}[\int^{\infty}_0
e^{-(a_1+\lambda)t}(a_1g^{21}+F(X^x_t))dt]
\end{equation*}
By Fatou lemma, we have
\begin{equation*}
G_2(\infty) = \varliminf_{x\to\infty}\mathbf{E}[\int^{\infty}_0
e^{-(a_1+\lambda)t}(a_1g^{21}+F(X^x_t))dt] \geq
\mathbf{E}[\int^{\infty}_0e^{-(a_1+\lambda)t}(a_1g^{21}+F(\infty))dt]>0
\end{equation*}
This is a contradiction. Thus, we have $\overline{x}^2\neq\infty$.

\item
Suppose that $G_0(\cdot)$ is the solution with linear growth to the
following ODE:
\begin{equation*}
-\mathcal{L}G_0(x)+(a_1+\lambda)G_0(x)=a_1g^{21}+F(x).
\end{equation*}
Since $a_1g^{21}+F(x)<a_1g^{21}+F(\infty)\leq 0$, the subsolution
property in Proposition \ref{propositioncompare} implies that
$G_0(x)\leq 0$ for $x\in(0,\infty)$. Thus $G_0(\cdot)$ is one
solution with linear growth to (\ref{s6}). Due to the uniqueness of
the solution to (\ref{s6}), we must have $G_2(x)=G_0(x)\leq 0$ for
$x\in(0,\infty)$, which means $\mathcal{S}^2 = (0,\infty)$.
\end{enumerate}
\end{proof}

The financial intuition behind Proposition \ref{p2} is that (1) when
the switching cost from a ``higher regime'' to a ``lower regime'' is
non-negative, it is unnecessary to switch; (2) if the switching cost
is negative and can ``compensate'' the loss due to switching to a
``lower regime'' in some state, the player would switch to a ``lower
regime'' at some level at Poisson arrival times; (3) if the profit
from switching cost exceeds the maximum loss from the ``net running
profit'', the player would switch to a ``lower regime'' from a
``higher regime'' at Poisson arrival times.

\appendix
\section{Appendix}

\subsection{Proof of Proposition \ref{PropositionBSDE}}
\noindent\textbf{Existence of solutions to the infinite horizon BSDE
system (\ref{IBSDE1})}

The idea is to truncate the infinite horizon BSDE system
(\ref{IBSDE1}) on $[0,\infty)$ to a finite horizon one on $[0,n]$
for any $n\geq 0$:
\begin{equation}\label{TBSDE1}
Y_t^i(n)=\int_t^nf_s^i(Y_s(n),Z_s(n))+\lambda\max\left\{0,
\mathcal{M}Y_s^i(n)-Y_s^i(n)\right\}ds-\int_t^{n}Z_s^i(n)dW_s.
\end{equation}
Then for $n\geq m\geq 0$, we consider the difference of two
truncated equations truncated on different time intervals $[0,n]$
and $[0,m]$:
\begin{align*}
Y_t^i(n)-Y_t^i(m)=&\
\int_t^nf_s^i(Y_s(n),Z_s(n))-f_s^i(Y_s(m),Z_s(m))+f_s^i(Y_s(m),Z_s(m))\mathbf{1}_{\{s\geq
m\}}ds\\
&+\int_t^n\lambda\left(\mathcal{M}Y_s^i(n)-Y_s^i(n)\right)^+-\lambda\left(\mathcal{M}Y_s^i(m)-Y_s^i(m)\right)^+
+\lambda\left(\mathcal{M}Y_s^i(m)-Y_s^i(m)\right)^+\mathbf{1}_{\{s\geq
m\}}ds\\
&-\int_t^{n}\left(Z_s^i(n)-Z_s^i(m)+Z_s^i(m)\mathbf{1}_{\{s\geq
m\}}\right)dW_s
\end{align*}
Note that $Y_t(m)=Z_t(m)=0$ on $\{t\geq m\}$. Hence, we have
\begin{equation*}
f_s^i(Y_s(m),Z_s(m))\mathbf{1}_{\{s\geq
m\}}=f_s^i(0,0)\mathbf{1}_{\{s\geq m\}},
\end{equation*}
and
\begin{equation*}
\lambda\left(\mathcal{M}Y_s^i(m)-Y_s^i(m)\right)^+\mathbf{1}_{\{s\geq
m\}}=Z_s^i(m)\mathbf{1}_{\{s\geq m\}}=0.\\
\end{equation*}

Now apply It\^o's formula to $e^{2at}|Y_t(n)-Y_t(m)|^2$,
\begin{align}\label{ITO}
&\ e^{2at}|Y_t({n})-Y_t({m})|^2\nonumber\\
=&\ -\int_t^{n}2a e^{2a s}|Y_s(n)-Y_s(m)|^2ds\nonumber\\
&\ +\int_t^n2e^{2a
s}(Y_s(n)-Y_s(m))^*\left(f_s(Y_s(n),Z_s(n))-f_s(Y_s(m),Z_s(m))+f_s(0,0)\mathbf{1}_{\{s\geq m\}}\right)ds\nonumber\\
&\ +\int_t^n2\lambda e^{2a
s}[Y_s(n)-Y_s(m)]^*\left((\mathcal{M}Y_s(n)-Y_s(n))^+-(\mathcal{M}Y_s(m)-Y_s(m))^+\right)ds\nonumber\\
&\ -\int_t^n2e^{2a s}(Y_s(n)-Y_s(m))^*(Z_s(n)-Z_s(m))dW_s\nonumber\\
&\ -\int_t^ne^{2a s}||Z_s(n)-Z_s(m)||^2ds.
\end{align}

By the monotone condition and Lipschitz condition in Assumption
\ref{Assumption1}, the second term on the RHS of (\ref{ITO}) is
dominated by
\begin{align}\label{ITO1}
&\int_t^n-2a_1e^{2as}|Y_s(n)-Y_s(m)|^2+2a_2e^{2as}|Y_s(n)-Y_s(m)|\times||Z_s(n)-Z_s(m)||ds\nonumber\\
&+\int_m^n2e^{2as}(Y_s(n)-Y_s(m))^*f_s(0,0)ds\nonumber\\
\leq&\int_t^n-2a_1e^{2as}|Y_s(n)-Y_s(m)|^2ds\nonumber\\
&+\int_t^n\delta_1^2a_2^2e^{2a
s}|Y_s(n)-Y_s(m)|^2+\frac{1}{\delta_1^2}e^{2as}||Z_s(n)-Z_s(m)||^2ds\nonumber\\
&+\int_m^n\delta_2^2a_2^2
e^{2as}|Y_s(n)-Y_s(m)|^2+\frac{1}{\delta_2^2a_2^2}e^{2a
s}|f_s(0,0)|^2ds
\end{align}
for any constants $\delta_1,\delta_2$, where we used the elementary
inequality $2ab\leq \frac{1}{\delta^2}a^2+\delta^2b^2$.

Similarly, the third term on the RHS of (\ref{ITO}) is dominated by
\begin{align}\label{ITO2}
&\int_t^n2\lambda
e^{2as}|Y_s(n)-Y_s(m)|\times|(\mathcal{M}Y_s(n)-Y_s(n))-(\mathcal{M}Y_s(m)-Y_s(m))|ds\nonumber\\
\leq&\int_t^n2\lambda
e^{2as}\left(|Y_s(n)-Y_s(m)|^2+|Y_s(n)-Y_s(m)|\times|\mathcal{M}Y_s(n)-\mathcal{M}Y_s(m)|\right)ds\nonumber\\
\leq&\int_t^n2\lambda
e^{2as}\left(1+\frac{1}{2}+\frac{d}{2}\right)|Y_s(n)-Y_s(m)|^2ds
\end{align}

By plugging (\ref{ITO1}) and (\ref{ITO2}) into (\ref{ITO}), and
choosing $\delta_1^2>1$ and $\delta_1^2+\delta_2^2<\delta$, and $a$
as in the structure condition (\ref{Structure}), we obtain
\begin{align*}
&e^{2at}|Y_t({n})-Y_t({m})|^2+(1-\frac{1}{\delta_1^2})\int_t^ne^{2a
s}||Z_s(n)-Z_s(m)||^2ds\\
\leq& \int_m^n\frac{1}{\delta_2^2a_2^2}e^{2a
s}|f_s(0,0)|^2ds-\int_t^n2e^{2a
s}(Y_s(n)-Y_s(m))^*(Z_s(n)-Z_s(m))dW_s
\end{align*}

Taking expectation at $t=0$, we have
$$\mathbf{E}\left[\int_0^ne^{2a
s}||Z_s(n)-Z_s(m)||^2ds\right]\leq
\frac{\delta_1^2}{(\delta_1^2-1)\delta_2^2a_2^2}\mathbf{E}\left[\int_m^ne^{2a
s}|f_s(0,0)|^2ds\right]\downarrow 0$$ as $m,n\uparrow\infty$. Hence,
$(Z(n))_{n\geq 0}$ is a Cauchy sequence in
$\mathcal{H}^2_{a}(\mathbb{R}^{d\times n})$, and converges to some
limit process, denoted as $Z$. On the other hand, taking supremum
over $t\geq 0$, and then taking expectation, we obtain
\begin{align*}
&\mathbf{E}\left[\sup_{t\geq 0}e^{2at}|Y_t({n})-Y_t({m})|^2\right]\\
\leq&\frac{1}{\delta_2^2a_2^2}\mathbf{E}\left[\int_m^ne^{2a
s}|f_s(0,0)|^2ds\right]+\mathbf{E}\left[\sup_{t\geq 0}
|\int_t^{\infty}2e^{2a
s}(Y_s(n)-Y_s(m))^*(Z_s(n)-Z_s(m))dW_s|\right]
\end{align*}

The standard argument by using the BDG inequality implies that the
martingale term is in fact uniformly integrable, so by taking
$m,n\uparrow\infty$, we deduce that $(Y(n))_{n\geq 0}$ is a Cauchy
sequence in $\mathcal{S}^2_{a}(\mathbb{R}^{d})$, and converges to
some limit process, denoted as $Y$. It is standard to check that
$(Y,Z)$ indeed satisfies (\ref{IBSDE1}), so in order to verify that
they are one solution to the infinite horizon BSDE system
(\ref{IBSDE1}), we only need to prove that
\begin{equation}\label{desiredconvergence}
\lim_{t\uparrow\infty}\mathbf{E}\left[e^{2at}|Y_t|^2\right]=0.
\end{equation}
Indeed, since $(Y(n))_{n\geq 0}$ is a Cauchy sequence in
$\mathcal{S}^2_a(\mathbb{R}^d)$, for any $\epsilon>0$, there exists
$n$ large enough such that
$$\mathbf{E}\left[e^{2at}|Y_t|^2\right]
\leq 2\mathbf{E}\left[e^{2at}|Y_t-Y_t(n)|^2\right]+
2\mathbf{E}\left[e^{2at}|Y_t(n)|^2\right]<2\epsilon+2\mathbf{E}\left[e^{2at}|Y_t(n)|^2\right].$$
Letting $t\uparrow\infty$ and noting $Y_t(n)=0$ for $t\geq n$, we
obtain the desired convergence (\ref{desiredconvergence}).\\

\noindent\textbf{Uniqueness of solutions to the infinite horizon
BSDE system (\ref{IBSDE1})}

%

The proof of uniqueness is similar to the proof of the existence, so
we only sketch it. Let $(Y,Z)$ and $(\tilde{Y},\tilde{Z})$ be two
solutions to the infinite horizon BSDE system (\ref{IBSDE1}). Denote
$\delta Y_t^i=Y_t^i-\tilde{Y}_t^i$, and $\delta
Z_t^{i}=Z_t^{i}-\tilde{Z}_t^{i}$. Then $(\delta Y,\delta Z)$
satisfies the following equation:
$$\delta Y^i_t=\delta Y_{T}^i+\int_t^T F_s^i(Y_s,Z_s)-F_s^i(\tilde{Y}_s,\tilde{Z}_s)ds-\int_t^T\delta Z^i_sdW_s,$$
where we denote
$$F_s^i(y,z)=f^i_s(y,z)+\lambda\max\{0,\mathcal{M}y^i-y^i\}$$
for any $(y,z)\in\mathbb{R}^d\times\mathbb{R}^{d\times n}$.

Apply It\^o's formula to $e^{2at}|\delta Y_t|^2$,
\begin{align}\label{ITO111}
e^{2at}|\delta Y_t|^2
=&\ e^{2aT}|\delta Y_{T}|^2-\int_t^{T}2a e^{2a s}|\delta Y_s|^2ds\nonumber\\
&\ +\int_t^T2e^{2a
s}(\delta Y_s)^*\left(F_s(Y_s,Z_s)-F_s(\tilde{Y}_s,\tilde{Z}_s)\right)ds\nonumber\\
&\ -\int_t^T2e^{2a s}(\delta Y_s)^*\delta Z_sdW_s-\int_t^Te^{2a
s}||\delta Z_s||^2ds.
\end{align}
Using the monotone condition and the Lipschitz condition in
Assumption \ref{Assumption1}, we get
\begin{align}\label{fundamentalinequality }
&e^{2at}|\delta Y_t|^2+(1-\frac{1}{\delta_1^2})\int_t^Te^{2a
s}||\delta Z_s||^2ds\nonumber\\
\leq&\ e^{2aT}|\delta
Y_T|^2-\left(2a+2a_1-\delta_1^2a_2^2-(d+3)\lambda\right)\int_t^Te^{2as}|\delta
Y_s|^2ds-\int_t^T2e^{2a s}(\delta Y_s)^*\delta Z_sdW_s.
\end{align}

Choosing $\delta_1^2=\delta$ and taking expectation on
(\ref{fundamentalinequality }) at $t=0$, we obtain
$$(1-\frac{1}{\delta})\mathbf{E}\left[\int_0^Te^{2a
s}||\delta Z_s||^2ds\right]\leq \mathbf{E}\left[e^{2aT}|\delta
Y_T|^2\right].$$ Since
$\lim_{T\uparrow\infty}\mathbf{E}\left[e^{2aT}|\delta
Y_T|^2\right]=0$, we conclude that
$$
\mathbf{E}\left[\int_0^{\infty}e^{2a s}||\delta
Z_s||^2ds\right]=0.$$

On the other hand, choosing $\delta_1^2=\delta$, taking supremum
over $t\geq 0$, and take expectation on (\ref{fundamentalinequality
}), we get
$$\mathbf{E}\left[\sup_{t\geq 0}e^{2at}|\delta Y_t|^2\right]\leq
\mathbf{E}\left[e^{2aT}|\delta Y_T|^2\right]+
\mathbf{E}\left[\sup_{t\geq 0} |\int_t^{T}2e^{2a s}(\delta
Y_s)^*\delta Z_sdW_s|\right].$$ Since
$\lim_{T\uparrow\infty}\mathbf{E}\left[e^{2aT}|\delta
Y_T|^2\right]=0$, and the martingale is uniformly integrable, we
conclude that
$$\mathbf{E}\left[\sup_{t\geq 0}e^{2at}|\delta Y_t|^2\right]=0.$$

\subsection{Proof of Proposition \ref{Propositionlineargrowth}}

For any switching strategy $u\in\mathcal{K}_i(0,\lambda)$, the
running profit of (\ref{optimalswitching3}) is dominated by
\begin{equation}\label{inequ1}
\mathbf{E}\left[\int_0^{\infty}e^{-a_1s}h^{u_s}(X_s^x)\right]\leq
C\mathbf{E}\left[\int_0^{\infty}e^{-a_1s}X_s^xds\right]=\frac{Cx}{a_1-b},
\end{equation}
where we used $h(0)=0$ and the Lipschitz continuity of $h(\cdot)$.
The switching cost of (\ref{optimalswitching3}) is dominated by
\begin{equation}\label{inequ2}
-\sum_{k=1}^{M}e^{-a_1T_k}g^{\alpha_{k-1},\alpha_k}\leq\max_{j=1,2}[-g^{ij}]
\end{equation}
for any $M\geq 1$, which can be proved by induction as in
\cite{LyVath}. Indeed, (\ref{inequ2}) obviously holds for $M=1$.
Suppose (\ref{inequ2}) holds for $M\geq 1$, we consider $M+1$. When
$g^{\alpha_M,\alpha_{M+1}}\geq 0$, (\ref{inequ2}) obviously holds as
$a_1\geq 0$. When $g^{\alpha_M,\alpha_{M+1}}\leq 0$, we have
\begin{align*}
-\sum_{k=1}^{M+1}e^{-a_1T_k}g^{\alpha_{k-1},\alpha_k}&
=-\sum_{k=1}^{M-1}e^{-a_1T_k}g^{\alpha_{k-1},\alpha_k}
-(e^{-a_1T_{M}}g^{\alpha_{M-1},\alpha_M}+e^{-a_1T_{M+1}}g^{\alpha_{M},\alpha_{M+1}})\\
&\leq
-\sum_{k=1}^{M-1}e^{-a_1T_k}g^{\alpha_{k-1},\alpha_k}-e^{-a_1T_M}(g^{\alpha_{M-1},\alpha_M}+g^{\alpha_{M},\alpha_{M+1}})\\
&\leq
-\sum_{k=1}^{M-1}e^{-a_1T_k}g^{\alpha_{k-1},\alpha_k}-e^{-a_1T_M}g^{\alpha_{M-1},\alpha_{M+1}}
\leq \max_{j=1,2}[-g^{ij}].
\end{align*}

Combining (\ref{inequ1}) and (\ref{inequ2}) gives us
$$\mathbf{E}\left[\int_0^{\infty}e^{-a_1s}h^{u_s}(X_s^x)ds-\sum_{k\geq
1}e^{-a_1T_k}g^{\alpha_{k-1},\alpha_k}\right]\leq
\frac{Cx}{a_1-b}+\max_{j=1,2}[-g^{ij}].$$ From the arbitrariness of
$u\in\mathcal{K}_i(0,\lambda)$, we obtain that the RHS of the above
inequality is the upper bound of $v^i(\cdot)$.

On the other hand, for $j=1,2$, choose a switching strategy as
$$u_s=i\mathbf{1}_{[0,T_1]}(s)+j\mathbf{1}_{[T_1,\infty)}(s).$$
Then, since $h\geq 0$,
$$\mathbf{E}\left[\int_0^{\infty}e^{-a_1s}h^{u_s}(X_s^x)ds-\sum_{k\geq
1}e^{-a_1T_k}g^{\alpha_{k-1},\alpha_k}\right]\geq
\mathbf{E}\left[-e^{a_1T_1}g^{ij}\right]=\frac{-\lambda}{a_1+\lambda}g^{ij}\geq
\frac{\lambda}{a_1+\lambda}\max_{j=1,2}\{-g^{ij}\}.$$ Once again,
from the arbitrariness of $u\in\mathcal{K}_i(0,\lambda)$, we obtain
that the RHS of the above inequality is the lower bound of
$v^i(\cdot)$, so $v^i(\cdot)$ has at most linear growth.

We conclude by showing the Lipschitz continuity of $v^i(\cdot)$.
Indeed, for $x,\bar{x}\in\mathbb{R}^+$,
\begin{align*}
|v^i(x)-v^i(\bar{x})|&\leq \sup_{u\in\mathcal{K}_i(0,\lambda)}
\mathbf{E}\left[\int_0^{\infty}e^{-a_1s}|h^{u_s}(X_s^x)-h^{u_s}(X_s^{\bar{x}})|ds\right]\\
&\leq\sup_{u\in\mathcal{K}_i(0,\lambda)}\mathbf{E}\left[\int_0^{\infty}
e^{-a_1s}C|X_s^x-X_s^{\bar{x}}|ds\right]\\
&=\mathbf{E}\left[\int_0^{\infty}
e^{-a_1s}C|x-\bar{x}|e^{(b-\frac12\sigma^2)s+\sigma W_s}ds\right]\\
&=\frac{C}{a_1-b}|x-\bar{x}|.
\end{align*}

\subsection{Proof of Proposition \ref{propositioncompare}}

We only prove the subsolution property, as the supersolution
property is similar. The proof relies on the comparison principle
for the ODE (\ref{ode1}) on a finite interval. For any
$x^0\in(0,\infty)$ and $\epsilon>0$, since $b<a_1$, we can choose
$p>0,q>1$ such that
$$\frac12\sigma^2p(p+1)-bp-a_1\leq 0.$$
$$\frac12\sigma^2q(q-1)+bq-a_1\leq 0.$$
With such $p,q$, we then choose $C_1,C_2,C_3>0$ such that
$$C_1(x^0)^{q}+\frac{C_2}{(x^0)^p}+C_3\leq \epsilon.$$
Now we consider the following auxiliary function
$$\underline{w}^i(x)=\underline{v}^i(x)-\left(C_1x^q+\frac{C_2}{x^p}+C_3\right).$$
Then we have
$$(-\mathcal{L}+a_1)\underline{w}^i(x)=(-\mathcal{L}+a_1)\underline{v}^i(x)+C_1x^q\left(\frac{1}{2}\sigma^2q(q-1)+bq-a_1\right)+\frac{C_2}{x^p}\left(\frac12\sigma^2p(p+1)-bp-a_1\right)-a_1C_3.
$$
Since $\underline{v}^i(\cdot)$ is the subsolution to the ODE
(\ref{ode1}) and the terms involving $p,q$ are negative by the
choice of $p,q$, we obtain that
$(-\mathcal{L}+a_1)\underline{w}^i(x)\leq 0$ for $x\in(0,\infty)$.

On the other hand, by the linear growth condition of
$\underline{v}^i(\cdot)$, there exits $\underline{x}$ and
$\overline{x}$ such that $x^0\in[\underline{x},\overline{x}]$ and
$\underline{w}^i(\overline{x})\leq 0$ and
$\underline{w}^i(\underline{x})\leq 0$. Then the comparison
principle for the ODE (\ref{ode1}) on a finite interval
$[\underline{x},\overline{x}]$ implies that
$\underline{w}^i(x^0)\leq 0$, and therefore,
$\underline{v}^i(x^0)\leq \epsilon$. Letting $\epsilon\downarrow 0$,
we obtain $\underline{v}^i(x^0)\leq 0$ for any $x^{0}\in(0,\infty)$.


\begin{figure}
[!htb]\label{Fig1}\centering
\includegraphics[totalheight=18cm]{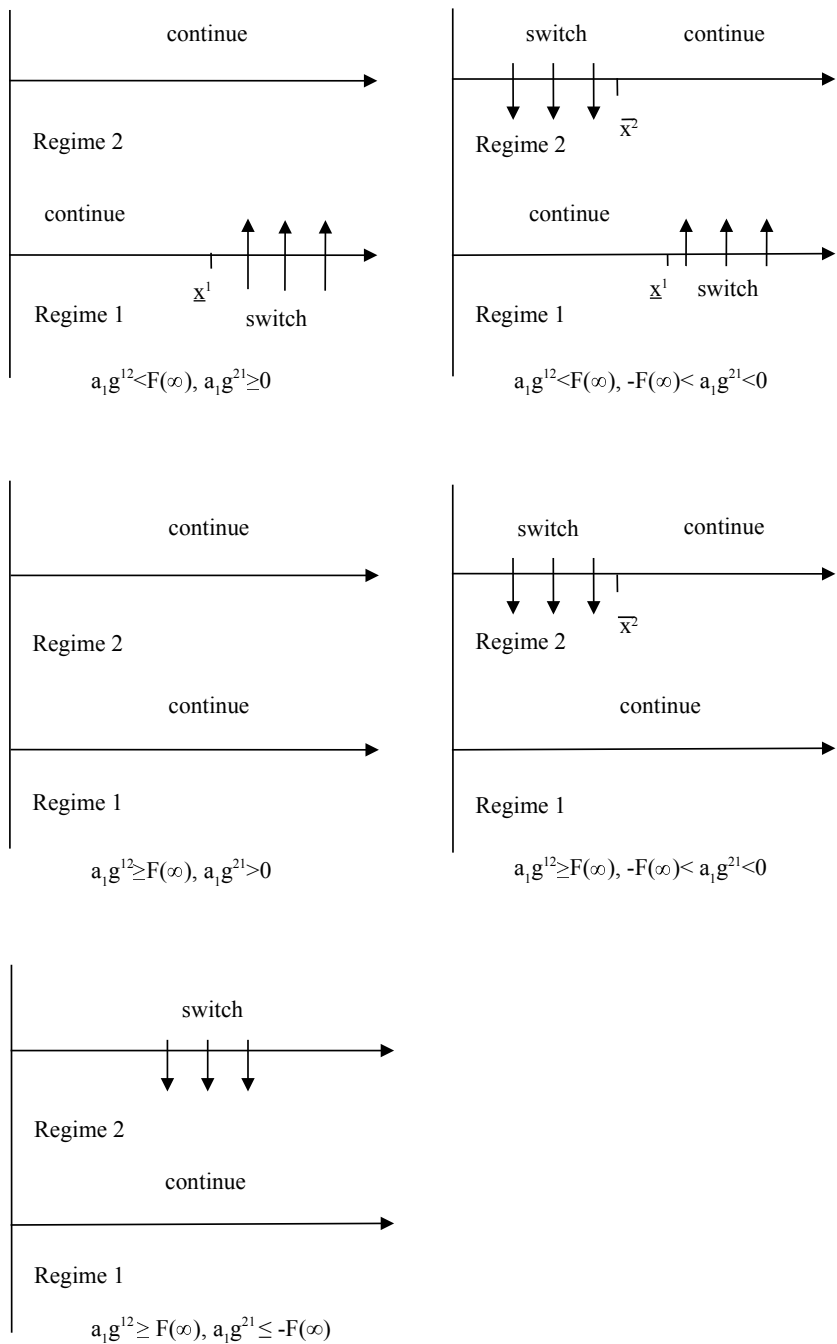}
\caption{The structure of switching regions}
\end{figure}


\end{document}